\documentclass[11pt]{amsart}
\usepackage{amssymb}
\usepackage{graphicx}
\usepackage[all,cmtip]{xy}
\newcommand{\s}{\sigma}

\newcommand{\SL}{{\mathcal{L}}}

\newcommand{\Z}{\mathbb{Z}}
\newcommand{\C}{\mathbb{C}}

\newcommand{\N}{\mathbb{N}}
\newcommand{\R}{\mathbb{R}}

\newcommand{\CP}{\mathbb{CP}}

\newcommand{\Diff}{\operatorname{Diff}}

\newtheorem{proposition}{Proposition}[section]
\newtheorem{theorem}[proposition]{Theorem}
\newtheorem{definition}[proposition]{Definition}
\newtheorem{lemma}[proposition]{Lemma}
\newtheorem{conjecture}[proposition]{Conjecture}
\newtheorem{corollary}[proposition]{Corollary}
\newtheorem{remark}[proposition]{Remark}

\begin{document}

\title{Geometric decompositions of almost contact manifolds}

\date{March, 2012}

\thanks{The author was supported by the Spanish National Research Project MTM2010-17389 and the ESF Network Contact and Symplectic Topology.}

%\author{Roger Casals}
%\address{Instituto de Ciencias Matem\'aticas CSIC-UAM-UC3M-UCM,
%C. Nicol\'as Cabrera, 13-15, 28049, Madrid, Spain}
%\email{casals.roger@icmat.es}

\author{Francisco Presas}
\address{Instituto de Ciencias Matem\'aticas CSIC--UAM--UC3M--UCM,
C. Nicol\'as Cabrera, 13--15, 28049, Madrid, Spain}
\email{fpresas@icmat.es}

%\begin{document}

\renewcommand{\theenumi}{\roman{enumi}}

\begin{abstract}
These notes are intended to be an introduction to the use of approximately holomorphic techniques in almost contact and contact geometry. We develop the setup of the approximately holomorphic geometry. Once done, we sketch the existence of the two main geometric decompositions available for an almost contact or contact manifold: open books and Lefschetz pencils. The use of the two decompositions for the problem of existence of contact structures is mentioned.
\end{abstract}

\maketitle
\tableofcontents

\section{Introduction} \label{sect:introduction}
Projective algebraic geometry is a field in which meaningful classification and existence questions of manifolds have been answered. Complete theories have been developed in the last two centuries: from the classification of algebraic curves already completed by Riemann, the study and classification of surfaces by the Italian school at the beginning of the 20th century, to the more recent high--dimensional analogues studied by means of Mori theory. There are two central ingredients in these theories:
\begin{itemize}
\item[-] The existence of algebraic curves in abundance in a projective variety.
\item[-] The theory of divisors: the algebraic understanding of the codimension one subvarieties of a projective variety.
\end{itemize}
Symplectic manifolds can be thought as topological generalizations of the projective varieties. In the projective setting the essential geometric object from which the theory is developed is the hyperplane divisor. In symplectic geometry the symplectic form is the topological analogue of this divisor: in a projective variety the Poincar\'e dual of the hyperplane section is the induced Fubini--Study symplectic form. Therefore, for a while, it was thought that the classification of symplectic manifolds could be achieved by the same methods as in the projective case. For that to work a correct generalization of the concept of algebraic curve and divisor had to be provided. The algebraic curves concept was generalized by the notion of pseudo--holomorphic curves introduced by Gromov \cite{Gr85} and has been shown to be central in the development of the symplectic topology. In real dimension $4$, a divisor coincides with an algebraic curve and this has been enough to push a meaningful theory in such a situation. There was left to procure an analogue of the Riemann--Roch theorem providing the existence of algebraic curves when topologically expected. In the nineties, C. Taubes showed how to handle this by introducing a relation with the Seiberg--Witten invariants \cite{Ta99}. From that point onwards a partial classification of $4$--dimensional symplectic manifolds has been achieved, e.g. see \cite{LM96}.\\

In higher dimensions a correct theory of divisors is lacking and probably it is not reasonable to expect it, since the symplectic geography problem in high dimensions is considerably wild, cf. \cite{Go95}. However, the particular case of very ample divisors was worked out by S. Donaldson in a series of foundational articles \cite{Do96, Do99}. The claim is that a theory of asymptotically very ample divisors can be developed in symplectic geometry, in other words very ample linear systems are available. The notion of ampleness is related to positivity, which holds due to the non--degeneracy of the symplectic form.  The implications of these results are the same as in projective geometry:
\begin{itemize}
\item[-] Bertini's theorem on the existence and genericity of smooth very ample divisors \cite{Do96}.
\item[-] Existence of symplectic Lefschetz pencils \cite{Do99} and associated symplectic invariants \cite{Do98,ADK04}.
\item[-] Connectedness of the space of very ample divisors \cite{Au97}.
\item[-] High--dimensional linear systems in the symplectic setting \cite{Au00}.
\end{itemize}
Maybe, the main conclusion is the existence of nice decompositions of a symplectic manifold in the same fashion as in the projective setting. This is not enough to classify though, but it provides a better understanding of the symplectic topology. In other words, a Lefschetz pencil is a clever way of trivializing a symplectic manifold.  Therefore, the implications of the existence result are the usual implications of a statement providing a combinatorial description of a geometric object: construction of solutions of equations \cite{DS03} and building blocks for the definition of new theories \cite{Se08}.\\

Contact geometry can be understood as a conformal analogue of the symplectic geometry and from this understanding the Donaldson techniques have been adapted to the contact setting. However, the general picture was initially far less clear since there is no classical analogue of the projective setting for contact manifolds. It turned out that there is one: the goal of these notes is to show its behaviour and the results it produces. The history developed as follows. The first attempt was to study the existence of codimension $2$ contact submanifolds on a general contact manifold \cite{IMP99}, this was non--achievable by the $h$--principle and it is the expected analogue of the Bertini's theorem in contact geometry. Not surprisingly, the contact picture was much more flexible than the symplectic one and it was shown that any codimension $2$ integer homology class on a closed contact manifold admits a smooth contact representative. The next two constructions to be worked out are the analogues of the:
\begin{itemize}
\item[-] Lefschetz pencil decomposition of a symplectic manifold \cite{Do99}.
\item[-] Decomposition of a symplectic manifold in terms of a very ample divisor and its Stein complementary \cite{Bi01}.
\end{itemize}
The equivalent of the Stein--divisor decomposition for a contact manifold is the open book decomposition constructed by Giroux and Mohsen \cite{Gi02,GM12}. There, the Stein manifold in which the symplectic manifold is trivialized is substituted by a $1$--parametric family of Stein manifolds, the so--called leaves of the open book, and the divisor becomes a codimension $2$ contact submanifold. The equivalent of the Donaldson's construction is straightforward and introduces the concept of a contact Lefschetz pencil \cite{Pr02}. The idea in that case is to produce a codimension $2$ fibration over the sphere whose fibers are contact manifolds, special singular fibers are also allowed corresponding to parametric holomorphic singularities.\\

The central question concerns the possible uses of these constructions, the essential feature being that these constructions are almost topological. In other words, they are $h$--principle achievable: there is no need for a contact structure in order to produce them, an almost contact structure is enough for them to exist. In case the contact structure could be recovered, they would produce an existence result in contact topology: any almost contact structure could be deformed to a contact one. This was highly unexpected a decade ago but nowadays it seems to be a reasonable statement in contact topology. The reason for the old perception is based on the fact that the almost symplectic condition does not imply the existence of a symplectic structure, which was proved in the late nineties. Hence it was believed to be a matter of time to find equivalent examples in the contact category. Nevertheless, the appearance of the previously mentioned decompositions gave support to the idea of almost contact implying contact being conceivable. The reason is that both the Stein--divisor decomposition and the Lefschetz pencil decomposition are not doable in the almost symplectic setting and they constitute an actual geometric obstruction to the existence of a symplectic structure on a general almost symplectic manifold.\\

This approach has been successful in dimension $5$, where any almost contact manifold has been proved to be contact \cite{CPP12} through the appropriate use of an almost contact pencil decomposition. The hope is that any of the two decompositions will eventually succeed to prove the existence of a contact structure in higher dimensions. Based on that, we detail in these notes the construction of the two decompositions.\\

\noindent The structure of the article reads as follows. In Section \ref{section:asymp} we establish the foundations of the approximately holomorphic theory for almost contact geometry. In Section \ref{sect:ob} we provide the argument of E. Giroux for the existence of open book decompositions adapted to a contact structure. In Section \ref{sect:pencil} we detail the construction of almost contact Lefschetz pencils after \cite{Pr02} and \cite{Ma09}. \\

\noindent {\bf Acknowledgements.} I want to thank Roger Casals by his proof-reading that has greatly improved the final version of this article. Special thanks also to Emmanuel Giroux who patiently explained to me the results contained in Section \ref{sect:ob}. Most of the ideas in these notes come from him. J.P. Mohsen kindly offered a copy of his thesis \cite{Mo01} to me. Discussions with Vincent Colin, Dishant Pancholi, Klaus Niederkr\"ueger, Thomas Vogel and Eva Miranda have been really helpful to write down these notes.

\section{Approximately holomorphic techniques} \label{section:asymp}
Let $M$ be a $(2n+1)$--dimensional smooth manifold. A global distribution $\xi\subset TM$ is said to be a contact structure if it admits a global $1$--form $\alpha\in \Omega^1(M)$ such that $\xi=\ker \alpha$ and $\alpha \wedge (d\alpha)^{n} >0$ everywhere. A contact manifold is a manifold with a contact structure. The $1$--form $\alpha$ defining the contact structure is said to be a contact form for the distribution.\\

\noindent In the literature this definition corresponds to the notion of a cooriented contact distribution, we will restrict ourselves to this case\footnote{Just once and for all it is important to mention that all the results in these notes can be easily adapted to the non--coorientable case. The essential point being that any non--coorientable contact manifold admits a coorientable double--cover. Therefore to study non--coorientable manifolds is reduced to study coorientable ones with free $\Z/2\Z$--actions. See \cite{IMP99} for details.}. Note that the contact condition does not strictly depend on the choice of the $1$--form $\alpha$, any other $1$--form $\alpha'= f\alpha$, with $f:M\longrightarrow\R^+$ a smooth function, also satisfies
$$ \alpha' \wedge (d\alpha')^n =f^{n+1} \alpha \wedge (d\alpha)^n>0. $$
Let us emphasize the topological features of a contact distribution. There are two topological objects appearing in the definition
\begin{enumerate}
\item The distribution $\xi$, a real codimension $1$ subbundle of $TM$.
\item The symplectic structure induced in the bundle $\xi$ by $d\alpha$. To be precise, only the conformal symplectic class is determined: a change of form $\alpha'=f\alpha$ as above does change the representatives from the symplectic bundle $(\xi, d\alpha)$ to $(\xi, fd\alpha)$.
\end{enumerate}
We therefore define an almost contact manifold as a $(2n+1)$--dimensional manifold $M$ with a codimension--1 cooriented\footnote{The normal bundle $TM/\xi$ of $\xi$ as a subbundle of $TM$ is trivial.} distribution $\xi$ and a conformally symplectic class on $\xi$, understood as an abstract bundle. A distribution admitting a conformally symplectic class is called an almost contact structure. The almost contact condition might be seen as the formal necessary condition for the existence of a contact structure. The long standing conjecture in contact topology is
\begin{conjecture}\label{conj:a}
Any almost contact structure on a manifold $M$ admits a deformation in its homotopy class of almost contact structures to a contact structure.
\end{conjecture}
In the case of open manifolds $M$, the result is true and it is one of the first applications of Gromov's $h$--principle, see \cite{Gr86}. The situation is not as established for closed manifolds. The conjecture was proven by R. Lutz \cite{Lu77} for $3$--dimensional closed manifolds. The classification of simply connected $5$--dimensional contact closed manifolds allowed H. Geiges \cite{Ge91} to also answer positively in these cases. The general $5$--dimensional case was recently proved by Casals et al., see \cite{CPP12}.  The conjecture remains open in general, some recent progress has been obtained by E. Giroux using the techniques described in Section \ref{sect:ob}.

\subsection{The quasi--contact category.}
An initial strategy in a geometric setting consists in understanding the implications of the $h$--principle; in our case we start with an almost--contact manifold. The further structure that the $h$--principle offers is provided in the following
\begin{definition}
A quasi--contact structure on a $(2n+1)$--dimensional manifold $M$ is a pair $(\xi,\beta)$  satisfying:
\begin{enumerate}
\item[-] $\xi$ is a cooriented distribution.
\item[-] $\beta$ is a $1$--form on $M$ such that $(\xi,d\beta)$ is a symplectic bundle.
\end{enumerate}
\end{definition}
Observe that the condition is stronger than the almost--contact one for $d\beta$ is necessarily closed, and not just a non--degenerate $2$--form. However, it is still weaker than the contact condition since $(\xi,\beta)$ inducing a contact structure would imply $\alpha=\beta$, with the previous notations. As previously mentioned, the quasi--contact condition can be reached through the $h$--principle, indeed one may show:
\begin{lemma} \label{lem:existe_quasi}
Any almost--contact structure admits a quasi--contact structure in its homotopy class of symplectic hyperplane fields.
\end{lemma}
For the proof see Lemma 2.2 in \cite{CPP12}. In the article \cite{CPP12} the definition of quasi--contact structure is given in a slightly more general setting, however no further applications are obtained and so we may concentrate in the more adapted definition above.\\

The fundamental property of quasi--contact manifolds is the closedness condition $d(d\beta)=0$. This is the precise piece of data we require to develop the theory of approximately holomorphic bundles: to begin with, a closed $2$--form topologically induces a complex line bundle. Let us start with the definitions: the pre--quantizable line bundle associated to the quasi--contact structure $(\xi,\beta)$ is the hermitian line bundle $L:=M\times \C$ with the choice of connection $\nabla_L= d - i\beta$.\\

A compatible almost complex structure for the quasi--contact structure $(\xi, \beta)$ is a compatible complex structure $J$ for the symplectic bundle $(\xi,d\beta)$. Also, a compatible metric for $(\xi,\beta,J)$ is any Riemannian metric $g$ such that
$$g(u,v)= d\beta(u, Jv), $$
for all $u,v \in \xi$ and such that $\ker d\beta$ is orthogonal to $\xi$. This amounts to a choice of a unitary vector field in $\ker d\beta$. Let us fix the unitary vector field orthogonal to $\xi$, it will be referred as the Reeb vector field $R$. We define the $1$--form
$$\alpha(v):=g(R,v),$$
that clearly satisfies $\ker \alpha=\xi$. We suppose that set of objects $(\xi, \beta, g)$ and the induced $R$ and $\alpha$ are given. For convenience, we also fix the following sequence of Riemannian metrics $g_k=k\cdot g$.\\

Let $E$ be a hermitian complex bundle with connection $\nabla$, we can split the connection along $\xi$ in its holomorphic and antiholomorphic parts since $\nabla$ restricted to $\xi$ is an operator between complex linear spaces and therefore admits a decomposition
$$ \nabla_{| \xi} = \partial + \bar{\partial}. $$
As explained in Section \ref{sect:introduction}, we should be able to produce symplectic and contact divisors. In analogy with the projective setting, the procedure Donaldson developed provides such divisors as vanishing loci of sections of a vector bundle. Instead of complex submanifolds from holomorphic sections we procure to obtain symplectic and contact submanifolds from asymptotically holomorphic sections. For a symplectic or contact structure to be induced in the vanishing locus the cut of the asymptotically holomorphic section with the base manifold has to satisfy certain transversality condition. Let us recall the following ideas from linear algebra:
\begin{definition}
A linear map $f: \R^n \longrightarrow \R^r$ is said to be $\varepsilon$--transverse to zero if it admits a right inverse of norm smaller than $\varepsilon^{-1}$.
\end{definition} 
\noindent There is a more geometric way of understanding the previous property
\begin{lemma}\label{lem:useful}
A linear map $f: \R^n \longrightarrow \R^r$ is $\varepsilon$--transverse to zero if and only if there exists an $r$--dimensional subspace $W \subset \R^n$ such that for any $w\in W$, we have
$$ |f(w)| \geq \varepsilon |w|.$$
\end{lemma}
\begin{proof}
If there exists a right inverse $g:\R^r\longrightarrow\R^n$, set $W=g(\R^r)$. This satisfies the required property. Conversely, suppose that such subspace $W\subset \R^n$ exists. Define $g$ as the right inverse map of the restriction $f:W \longrightarrow \R^r$, which exists since $dim_\R(W)=r$.
\end{proof}
\noindent The linear condition required for the tangent bundle of the submanifold to be a symplectic subbundle can be stated as
\begin{lemma} $($Proposition 3 in \cite{Do96}$)$ \label{lem:antilin}
Let $f: \C^n \to \C^r$ be an $\R$--linear map $\varepsilon$--transverse to zero. Suppose there exists a $\delta>0$, depending only on $\varepsilon$, such that the antiholomorphic part of $f$ satisfies
$$ |f^{0,1}| \leq \delta,$$
then $\ker f$ is a symplectic subspace of $\C^n$.
\end{lemma}
\noindent The proof is based on the fact that the condition $f^{0,1}=0$ implies that the subspace is complex and therefore symplectic for the compatible symplectic structure and the symplectic condition is open. We are in position to describe the suitable transversality condition:

\begin{definition}
Let $E\longrightarrow M$ be a hermitian complex bundle with connection $\nabla$. A section $s: M\longrightarrow E$ is said to be $\varepsilon$--transverse to zero along $\xi$ if $\forall x\in M$ any of the following conditions hold:
\begin{itemize}
\item[-] $|s(x)|>\varepsilon$,
\item[-] $\nabla_{\xi}(s)(x): \xi_x \to E_x$ is $\varepsilon$--transverse to zero.
\end{itemize}
\end{definition}

A submanifold $SÊ\stackrel{e}{\hookrightarrow} (M, \xi, \beta)$ is said to be quasi--contact if $\xi$ is everywhere transverse to $S$ and $(e^*(\xi), e^* \beta)$ is a quasi--contact structure on $S$. Let us provide a simple way to decide whether the zero locus of a section is a quasi--contact submanifold:
\begin{lemma} \label{lem:subquasi}
For any $\varepsilon>0$, there exists a $\delta>0$ such that if $s: M\longrightarrow E$ is $\varepsilon$--transverse to zero along $\xi$ and $|\bar{\partial} s| \leq \delta$, then the zero set $Z(s)$ is a smooth quasi--contact submanifold of $M$.
\end{lemma}
\begin{proof} Denote by $r$ the rank of the complex bundle $E$. The $\varepsilon$--transversality along $\xi$, in particular, implies that the section is transverse to zero in the usual sense, i.e. for any $x\in Z(s)$, the linear map $\nabla s(x)$ is surjective. Therefore, the set $Z(s)$ is a smooth submanifold of dimension $2(n-r)+1$. The transversality along $\xi$ further implies that the submanifold $Z(s)$ is transverse to $\xi$, since for any point $x\in Z(s)$ we have that the induced distribution $e^*(\xi)= \ker \nabla_{\xi} s(x)$ on $x$ has real dimension $2(n-r)$.\\

\noindent It is left to verify that $e^*(\xi)= \ker \nabla_{\xi} s$ is symplectic everywhere, but for a suitable choice of $\delta$ we are in the hypothesis of the Lemma \ref{lem:antilin}.
\end{proof}

\noindent Note that $\varepsilon$--transversality is a $C^1$--stable notion. Indeed, let $s: M\longrightarrow E$ be a section $\varepsilon$--transverse along $\xi$ and $s_\delta:M\longrightarrow E$ a perturbative section satisfying that $|s_\delta|_{C^1} \leq \delta$. Then the perturbed section $s+s_\delta$ still is $(\varepsilon-c\delta)$--transverse to zero along $\xi$ for some universal constant $c>0$. Thus, $C^1$--perturbations do not destroy the estimated transversality along $\xi$.\\

Let us define a central notion in the approximately holomorphic techniques, these are also referred as asymptotically holomorphic techniques since the produced objects acquire a holomorphic behaviour at the limit. Given a line bundle $L$ on $M$ constructed as above, we can associate to the hermitian bundle $E$ the following sequence of bundles $E_k:= E \otimes L^{\otimes k}$ for $k\in \N$. This is related to the twisting sheaf in projective geometry, allowing to shift a coherent sheaf to an affine behaviour. In quasi--contact geometry, our aim is to produce the following objects:
\begin{definition}
A sequence of sections $s_k:M \longrightarrow E_k$ is $C^r$--asymptotically holomorphic if the following estimates hold
$$ |s_k|=O(1), \, |\nabla^l s_k|=O(1), \, |\nabla^{l-1} \bar{\partial} s_k|=O(k^{-1/2}),\quad\forall l\leq r,$$
the norms being measured with respect to the $g_k$--metric.
\end{definition}
\noindent The index $l$ will be omitted if it is clear from the context. As a consequence of the previous discussion we conclude:
\begin{corollary}
Let $s_k: M\longrightarrow E_k$ be an asymptotically holomorphic sequence of sections $\varepsilon$--transverse to zero along $\xi$. For $k$ large enough, the set $Z(s_k)$ is a smooth quasi-contact submanifold.
\end{corollary}
\noindent The existence of $\varepsilon$--transverse asymptotically holomorphic sections is partially guaranteed due to the following:
\begin{theorem} \label{thm:Ber}
Let $E$ be a vector bundle, $\delta>0$ and $s_k: M\longrightarrow E_k$ be an asymptotically holomorphic sequence of sections. There exists a constant $\varepsilon>0$ and an asymptotically holomorphic sequence of sections $\sigma_k:  M\longrightarrow E_k$ such that they are $\varepsilon$--transverse to zero along $\xi$ and $|\sigma_k-s_k|_{C^2}\leq \delta$.
\end{theorem}
\noindent We will give an overview of the proof of this Theorem in the rest of this Section. However, it is just the generalization to the quasi--contact category of the main result in \cite{IMP99}. The result implies the existence of quasi--contact divisors with prescribed topology:
\begin{corollary}
Fix an integer homology class $A\in H_{2n-1}(M,\Z)$, there exists a smooth quasi--contact submanifold $S$ representing it.
\end{corollary}
\begin{proof}
Let $\gamma\in H^2(M, \Z)$ be the Poincar\'e dual of $A$. Construct a hermitian line bundle $L$ with $c_1(L)=\gamma$. Apply the Theorem \ref{thm:Ber} to the sequence $L_k$. Since $c_1(L_k)=c_1(L)=\gamma=PD(A)$, any submanifold $Z(s_k)$, for $k$ large enough, fulfills the requirements.
\end{proof}

The previous construction can also be made relative to a complex distribution. Let $s_k: M \longrightarrow E \otimes L^{\otimes k}$ be a $C^r$--asymptotically holomorphic sequence of sections and $D\subset \xi$ any fixed complex distribution, it is simple to verify that $\partial_D s_k: M \longrightarrow D^*\otimes E\otimes L^{\otimes k}$ is a $C^{r-1}$--asymptotically holomorphic sequence of sections. There is also an existence result for this case:

\begin{theorem} \label{thm:Ber2}
Let $V$ be a line bundle, $D\subset \xi$ complex distribution and $\delta>0$. Consider $s_k: M \longrightarrow V \otimes L^{\otimes k} \otimes \C^2$ an  asymptotically holomorphic sequence of sections. There exists a constant $\varepsilon$ and an asymptotically holomorphic sequence of sections $\sigma_k=(\sigma_{k,0}, \sigma_{k,1}): M \longrightarrow V \otimes L^{\otimes k} \otimes \C^2$ such that $|s_k- \sigma_k|_{C^2}\leq \delta$ and $\partial_D \sigma_{k,0} \otimes \sigma_{k,1} -  \sigma_{k,0} \otimes \partial_D \sigma_{k,1}$ are $\varepsilon$--transverse to zero along $\xi$.
\end{theorem}

The previous results also hold for the particular setting in which $\alpha=\beta$, systematically replacing the word quasi--contact by contact. This is the formulation found in \cite{IMP99, Pr02}. However, the proofs remain practically unchanged in this more general setting. D. Mart\'inez--Torres has provided a general theory for the quasi--contact case in different articles \cite{IM04, Ma09}. In his notation the quasi--contact structures are called $2$--calibrated structures.  This Section is intended to provide a short version of \cite{Ma09} centered just in $0$ and $1$--dimensional linear systems; in that article the general theory for $r$--dimensional linear systems is worked out.  The essential problem to overcome is to obtain transversality for sequences of sections, there are two available techniques to do it:
\begin{itemize}
\item[-] The one developed in \cite{IMP99}, strictly working in the quasi--contact manifold itself.
\item[-] The relative version developed in \cite{Mo01} in which we embed the quasi--contact structure in a symplectic manifold and the transversality is achieved there.
\end{itemize}
We shall use the second alternative to deduce Theorem \ref{thm:Ber}. The definitions required for the symplectization setting are now provided.
\subsection{The symplectization of a quasi--contact structure.}
We define the symplectization of a quasi--contact structure $(M, \xi= \ker \alpha ,d\beta)$ as the symplectic manifold $S(M):= M \times [-\tau, \tau]$ with the form $\omega_S: = d(\beta +t\alpha)$, where $\tau>0$ is chosen small enough so that the form $\omega_S$ symplectic. 
We can construct the required data as restrictions of objects in the symplectization. In particular, we will need:
\begin{enumerate}
\item Any choice of compatible $J$ can be extended to an almost--complex structure $\hat{J}$ on $TS(M)= \xi \oplus R \oplus \partial_t$ by declaring $J \partial_t =R$. This almost--complex structure is compatible with $\omega_S$.
\item The associated Riemannian metric $\hat{g}=\omega_S(\cdot,\hat{J}\cdot)$ extends $g$ setting $\partial_t$ to be orthornormal to $TM \subset TS(M)$. Also, we define $\hat{g}_k=k \hat{g}$.
\item The prequantizable bundle on the quasi--contact manifold is the restriction of the prequantizable bundle over $S(M)$ defined as the trivial bundle $L=S(M) \times \C$ with connection $\nabla=d -i \hat{\beta}$, for the primitive form $\hat{\beta}= \beta+t\alpha$.
\end{enumerate}
These choices are to be assumed in the following discussion. Thus, we may use the definitions of transversality and asymptotically holomorphic sequences in the symplectic case, see \cite{Au97}. We briefly recall them; given a hermitian bundle $E$ over $S(M)$, denote $E_k:=E\otimes L^{\otimes k}$. The fundamental notion in the symplectic case is contained in the following:
\begin{definition}
A sequence of sections $s_k:S(M)\longrightarrow E_k$ is $C^r$--asymptotically holomorphic if the following estimates hold,
$$ |s_k|=O(1), \, |\nabla^l s_k|=O(1), \, |\nabla^{l-1} \bar{\partial} s_k|=O(k^{-1/2}),\quad\forall l\leq r,$$
the norms being measured with respect to the $\hat{g}_k$--metric.
\end{definition}
The splitting of the connection in holomorphic and antiholomorphic parts is a consequence of the usual decomposition $\nabla= \partial + \bar{\partial}$. The transversality condition is analogously stated as
\begin{definition}
Let $E\longrightarrow S(M)$ be a hermitian complex bundle with connection $\nabla$. A section $s: S(M)\longrightarrow E$ is said to be $\varepsilon$--transverse to zero over $M$, if $\forall x\in M \times \{ 0\}$ any of the following conditions hold:
\begin{itemize}
\item[-] $|s(x)|>\varepsilon$,
\item[-] $\nabla_{M}(s)(x): (T(M \times \{Ê0Ê\}))_x \to E_x$ is $\varepsilon$--transverse to zero.
\end{itemize}
\end{definition}
\noindent Observe that an asymptotically holomorphic sequence of sections in the symplectization $s_k: S(M)\longrightarrow E_k$ restricts to $M \times \{ 0 \}$ as an asymptotically holomorphic sequence of sections $e^*s_k$. It is less clear that transversality along the quasi--contact distribution can be also achieved. Let us prove the following
\begin{proposition} \label{ref:compare}
Let $s: S(M)\longrightarrow E$ be an $\varepsilon$--transverse section over $M$. Assume that $|\bar{\partial} s|\leq \delta$ for some $\delta>0$ small enough and depending only on $\varepsilon$. Then the restriction of the section
$$e^*s: M\longrightarrow e^* E$$ is an $\frac{\varepsilon}{2}$--transverse section along $\xi$.
\end{proposition} 
\begin{proof}
This is essentially a linear algebra question. Let $2r=rk(E)$ and fix a point $x\in M$ for which $|s(x)|\leq \varepsilon$, then the linear map $$\nabla_{M} s(x): \xi_x \oplus \langle R \rangle \to E_x$$ is $\varepsilon$--transverse to zero. Define $\varepsilon'=\frac{3}{4}\varepsilon$. Let $\delta>0$ be small enough such that $f=\partial s(x)$ is $\varepsilon'$--transverse to zero. Thus, there exists a right inverse of norm smaller than $(\varepsilon')^{-1}$. By Lemma \ref{lem:useful}, this implies that there exists a $2r$--dimensional subspace $W\subset \xi_x \oplus \langle R \rangle$ such that for any $w\in W$,
$$ |f(w)| \geq \varepsilon'|w|.$$
If $W \subset \xi_x$ we are done. Otherwise, define $V= W \cap \xi_x$ and let $U\subset V$ be a Lagrangian $r$--dimensional subspace. Consider $\hat{U}=f(U)$, then we obtain the splitting $E_x = \hat{U} \oplus i\hat{U}$. The map $f$ restricted to the subspace $U_{\C}=U \oplus JU \subset \xi$ satisfies
$$|f(u)|^2= |f(u_1)+if(u_2)|^2 = |f(u_1)|^2+|f(u_2)|^2 \geq (\varepsilon')^2|u|^2,$$
where $u=u_1 +Ju_2\in U_{\C}$. Therefore $f_{|\xi}$ is $\varepsilon'$--transverse to zero. Again, for fixed $\delta>0$ small enough, the linear map $\nabla_{\xi} s(x)$ is $\frac12 \varepsilon$--transverse to zero.
\end{proof}
\noindent Since we may achieve transversality with respect to the distribution, to conclude the proof of Theorem \ref{thm:Ber} we must ensure the existence of uniformly transverse asymptotically holomorphic sections. For that purpose, we refer to the main result in \cite{Mo01}:
\begin{theorem} $($Mohsen$)$ \label{thm:moh}
Let $(M, \omega)$ be a symplectic manifold of integer class. Fix a compatible almost complex structure $J$, a closed submanifold $S$ and a hermitian vector bundle $E$. Then for any asymptotically holomorphic sequence of sections $s_k: M\longrightarrow E\otimes L^{\otimes k}$ and for any $\delta>0$, there exists a $C^3$--asymptotically holomorphic sequence of sections $\sigma_k:M \to E\otimes L^{\otimes k}$ satisfying
\begin{enumerate}
\item[-] $|\sigma_k-s_k|_{C^2} \leq \delta$,
\item[-] The sequence $\sigma_k$ is $\varepsilon$--transverse to zero in $M$, for some uniform constant $\varepsilon>0$ not depending on $k$.
\end{enumerate}
\end{theorem}
\noindent It is now immediate to conclude the existence of transverse asymptotically holomorphic sections:\\

\noindent {\bf Proof of Theorem \ref{thm:Ber}.}
Let us describe the elements appearing in the hypothesis of Theorem \ref{thm:moh}. The symplectization $(S(M), d(\beta+t\alpha))$ will be the symplectic manifold. The submanifold will be the quasi--contact manifold, hence $S=M \times \{ 0 \}$. Finally, pull--back the vector bundle $E\longrightarrow M$ to a bundle in $S(M)$, still denoted $E$.  As a consequence of the theorem applied to the constant sequence $s_k=0$ we obtain a $C^2$--small sequence $\sigma_k: S(M) \longrightarrow E\otimes L^{\otimes k}$ which is $\varepsilon$--transverse to zero in $M$. After Proposition \ref{ref:compare} the sequence $e^* \sigma_k:M \to E\otimes L^{\otimes k}$ is $\frac12\varepsilon$--transverse to $\xi$.
\hfill $\Box$\\

\noindent To conclude Theorem \ref{thm:Ber2} we have to slightly generalize Theorem \ref{thm:moh} to allow certain control for the derivative of the quotient along the complex distribution. The precise statement we require is the following
\begin{theorem}\label{thm:mart}
Let $(M, \omega)$ be a symplectic manifold of integer class. Fix a compatible almost complex structure $J$, a closed submanifold $S$, a hermitian line bundle $V$ and a complex distribution $D \subset TS$ over the submanifold. Then, for any asymptotically holomorphic sequence of sections $s_k: M\longrightarrow V\otimes L^{\otimes k} \otimes \C^2 $ and for any $\delta>0$, there exists an asymptotically holomorphic sequence of sections $\sigma_k=(\sigma_{k,0}, \sigma_{k,1}):M \longrightarrow V\otimes L^{\otimes k} \otimes \C^2 $ satisfying
\begin{itemize}
\item[-] $|\sigma_k-s_k|_{C^2} \leq \delta$,
\item[-] $\partial_D \sigma_{k,0} \otimes \sigma_{k,1} - \sigma_{k,0} \otimes \partial_D \sigma_{k,1}$ is $\varepsilon$-transverse to zero in $M$, for some uniform constant $\varepsilon>0$ not depending on $k$.
\end{itemize}
\end{theorem}
\noindent The proof of this Theorem, in a more general case, can be found in \cite{Ma09}. Finally, Theorem \ref{thm:Ber2} is an immediate consequence of this result. \\

\section{Open books in contact geometry}\label{sect:ob}
\noindent This section develops the open book decomposition mentioned in Section \ref{sect:introduction}. We begin with the main definition:
\begin{definition}
Let $M$ be a smooth closed manifold. A pair of objects $(B, \pi)$ is called an open book decomposition if they satisfy:
\begin{itemize}
\item[-] $B$ is a codimension--$2$ closed submanifold.
\item[-] $\pi: M\setminus B \to S^1$ is a submersion.
\item[-] The normal bundle of $B$ is trivial and there exists a tubular neighborhood $U$ with a trivializing diffeomorphism $\phi: B \times B^2(\delta) \to U$ such that
$$(\pi \circ \phi)(p, r, \theta)= \theta,$$
where $p\in B$ and $(r,\theta)$ are polar coordinates in $B^2(\delta)$.
\end{itemize}
The divisor $B$ is referred as the binding. The closure of the fibers of $\pi$ in $M$ are called the pages of the open book $(B,\pi)$.
\end{definition}

\begin{figure}[ht]
\includegraphics[scale=0.42]{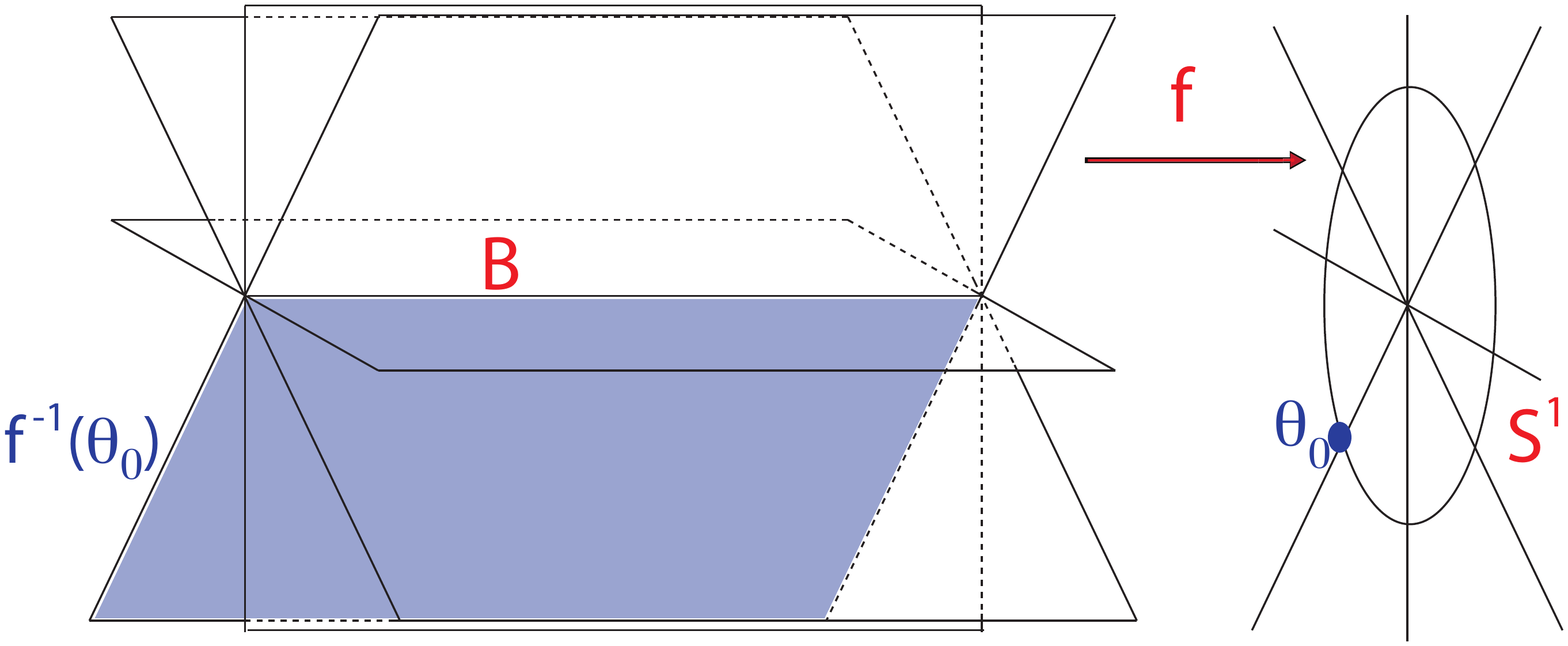}
\caption{Open book close to the binding.}
\end{figure}

Let us describe an equivalent construction. Consider a smooth manifold $P$ with boundary $B=\partial P$, and $\Psi:P \to P$ a diffeomorphism restricting to the identity close to the boundary. We then construct a closed manifold $M$ from the pair $(P,\Psi)$: as a topological space $M= (P \times [0,1])/\sim$ where the equivalence relation is defined as
$$ (p,t) \sim (q,s) \Longleftrightarrow \left\{ \begin{array}{l} p=q \in B, \\ {\rm or} \\ p=\Psi(q)\mbox{ and } t=0, s=1. \end{array}  \right. $$
This produces a manifold. Indeed, before quotienting it is certainly a manifold and then the quotient can be understood as a two--step process. In the first step $P \times \{ 0 \}$ and $P \times \{ 1 \}$ are identified by means of the diffeomorphism $\Psi$ to produce a manifold $P_\Psi$ that fibers over $S^1$, its boundary is diffeomorphic to $B \times S^1$. Secondly, in order to obtain the collapse from the first condition in the equivalence relation, we fill the boundary of $P_\Psi$ with $B\times D^2$. This produces a smooth manifold $M$ without boundary. Define the map
\begin{eqnarray*}
\pi: ((P \times [0,1])/\sim) \setminus B & \longrightarrow & S^1 \\
(p,t) & \longmapsto & t.
\end{eqnarray*}
Then $(B,\pi)$ is an open book decomposition of $M$ with pages diffeomorphic to $P$. Conversely, given an open book decomposition $(B, \pi)$ of a manifold $M$, we may recover $P=\overline{\pi^{-1}(0)}$ and $\Psi$. For the diffeomorphism, consider a connection for the fibration $\pi:M \setminus B \to S^1$, thus providing a notion of parallel transport, and then $\Psi\in\Diff(P)$ is obtained as the time--$1$ flow of the lifting of $\partial_t$ with respect to the chosen connection. Hence, we can define an open book decomposition either by providing the pair $(B, \pi)$ or the pair $(P, \Psi)$.\\

\noindent The notion of an open book decomposition is essentially topological. We now follow E. Giroux \cite{Gi02} to relate it with contact geometry. Given a contact form $\alpha$ for a contact structure, let $R=R_\alpha$ be the unique vector field such that $d\alpha(R,\cdot)=0$ and $\alpha(R)=1$.
This is called the Reeb vector field of $\alpha$. The interaction between contact geometry and open book decompositions is based on the following
\begin{definition}
Let $(M,\xi)$ be a contact manifold. A contact form $\alpha$ supports an open book decomposition $(B, \pi)$ if:
\begin{enumerate}
\item[-] $(B, \alpha_B=\alpha_{|B})$ is a contact submanifold.
\item[-] The Reeb vector field $R$ is positively transverse to the projection $\pi$, i.e. $d\pi(R) >0$ everywhere, and tangent to the submanifold $B$.
\end{enumerate}
\end{definition}
\noindent Given a fixed contact structure and a supporting contact form, the open book is said to be adapted to the contact structure through the contact form. The open book is called adapted to a contact structure if it is adapted through a contact form inducing the given contact structure. The requirements in the definition have the following implications:
\begin{itemize}
\item[-] The pages $P_t= \overline{\pi^{-1}(t)}$ inherit an exact symplectic structure provided by the restriction of $d\alpha$.
\item[-] The associated flow $\Psi$ is a symplectomorphism since the generating vector field is of the form $X=f\cdot R$ and $\SL_R d\alpha =0$.
\item[-] The boundary of any page is $\partial P_t=B$, and it is of convex type with respect to the symplectic structure $(d\alpha)|_{P_t}$.
\end{itemize}
Recall that an exact symplectic manifold $(M, \omega=d \alpha)$ has a boundary of convex type with respect to the Liouville form $\alpha$ if the associated Liouville vector field $X$, defined by $\alpha= i_X d \alpha$, is tranverse to the boundary of $M$ and points outwards. Convexity is a relevant property in procedures such as gluing or filling constructions and has a fundamental role in the understanding of Conjecture \ref{conj:a}. The first two assertions are readily seen to hold, let us detail the third statement:
\begin{lemma} \label{lem:special}
Let $(B,\pi)$ be an open book decomposition adapted to $(M,\ker\alpha)$. There exists a neighborhood $U$ of the binding $B$ and a trivializing diffeomorphism $\psi: B \times B^2(\delta) \to M$ such that
$$\psi^*\alpha= g\cdot(\alpha_{|B} + r^2d\theta),$$
where $g:B \times B^2(\delta) \longrightarrow \R^+$ satisfies $\partial_r g<0$ for $r>0$.
\end{lemma}
\begin{proof}
Consider the trivializing map $\phi: B \times B^2(\delta) \to U$ provided by the definition of an open book decomposition. Note that for $\delta>0$ small enough the fibers are contact submanifolds. Let $\pi_2: B\times B^2(\delta) \longrightarrow B^2(\delta)$ be the projection onto the second factor, then the projection
$$\pi_U= \phi^{-1} \circ \pi_2: U \to B^2(\delta)$$
is a contact fibration in the sense of \cite{Pr07}. As such, there is an associated contact connection. Certainly, at a point $p\in U$ the vertical subspace is $V_p= \ker d\pi_U(p)$. Since the fiber is a contact submanifold, $(\xi_B)_p=V_p \cap \xi_p$  is a symplectic subspace of $(\xi_p, d\alpha_p)$ and therefore we may define the horizontal subspace as the symplectic orthogonal $H_p= (\xi_B)_p^{\perp d\alpha_p}$. This defines a contact connection for the contact fibration. In particular, the induced parallel transport is by contactomorphisms. We use this connection to suitably trivialize the fibration $\pi_U$. Lifting the radial vector field $r\partial_r$ on the disk provides a flow on $U$: the associated contactomorphism from the central fiber $\pi_U^{-1}(0)$ to the general fiber $\pi_U^{-1}(r,\theta)$ will be denoted by $\Phi_{(r,\theta)}$. The appropriate trivialization is provided by the contactomorphism
\begin{eqnarray*}
\Phi: B \times B^2(\delta) & \longrightarrow & B \times B^2(\delta) \\
(p,r, \theta) & \longmapsto & (\Phi_{(r,\theta)}(p),r, \theta).
\end{eqnarray*}
The composition $\widetilde{\Phi}=\Phi \circ \phi$ satisfies
$$\widetilde{\Phi}^*\alpha = \widetilde{g} \cdot (\alpha_B +rH(p,r,\theta)d\theta),$$
where $\widetilde{g}$ is a strictly positive smooth function and $H$ is a function with the following properties:
\begin{enumerate}
\item[-] The identity being induced in the central fiber, $H(p,0,0)=0$.
\item[-] After the contact condition, $\partial_r (rH)>0$ in $r>0$.
\item[-] It achieves a radial minimum in the central fiber, $\partial_r H(p,0,0)>0$.
\end{enumerate}
In order to suppress the $H$ factor we further compose with
\begin{eqnarray*}
f: B \times B^2(\delta) & \longrightarrow & B \times B^2(\delta) \\
(p,r, \theta) & \longmapsto & (p,H, \theta),
\end{eqnarray*}
which is injective for $r$ small enough. We then obtain the diffeomorphism
$$\psi=\widetilde{\Phi} \circ f: B \times B^2(\delta) \to U$$ satisfying $\psi^*\alpha= g\cdot(\alpha_{|B} + r^2d\theta)$,
for some positive function $g$. Denote this form by $\alpha_g$, it remains to verify that the radial derivative of $g$ is negative. Let us express this in terms of the Reeb vector fields. Note that the open book map restricts as $(\pi \circ \psi)(p, r, \theta)= \theta$ and thus $R_\alpha$ satisfies
\begin{equation}\partial_\theta(\psi^* R_\alpha)>0,\quad \mbox{for }r>0\label{eq:goodR},
\end{equation}
since the set $\{r=0 \}$ is the binding $B$ in these coordinates. This condition implies $\partial_rg<0$. Indeed, decompose the Reeb vector field $R_g$ of $\alpha_g$ as $$\psi^*R=R_g = V +b \partial_r +c \partial_{\theta},\quad\mbox{ for some }V\in \Gamma(TB)\mbox{ and }b,c \in \R.$$
Condition (\ref{eq:goodR}) translates into $c>0$. The symplectic form is written as
$$ d\alpha_g= dg \wedge (\alpha_B+r^2 d\theta) + g(d\alpha_{B} +2rdr\wedge d\theta).$$
and from the defining equations of the Reeb vector field we obtain
$$0=d\alpha_g(R_g, \partial_r)= - \partial_rg \cdot \frac{1}{g}-crg.$$
Consequently, the condition on $g$ is verified as
$$c = -\partial_rg\cdot \frac{1}{g^2r} > 0\Longleftrightarrow \partial_rg<0.$$
\end{proof}
\noindent The third assertion regarding the convexity of the boundary can be deduced as follows:
\begin{corollary}
Let $(B,\pi)$ be an open book decomposition supported by $(M,\alpha)$. The binding $B$ is a convex boundary of any page $P_t=\overline{\pi^{-1}(t)}$ with respect to the Liouville vector field associated to the Liouville form $\alpha|_{P_t}$.
\end{corollary}
\begin{proof}
This is a computation close to the boundary, as such we may use the trivialization model provided in Lemma \ref{lem:special}. In this chart a page is defined as the set
$$ \widetilde{P}_t= \psi^{-1}(P_t)= \{(p,r, \theta)\in B \times B^2(\delta):\quad r>0,\quad \theta=t \}.$$
In these coordinates the Liouville vector field $X$ is given by the equation
$$ (\alpha_g)_{|P_t} = i_X (d \alpha_g)_{|P_t},$$
and the solution can be explicitly written as
$$ X= \left( \partial_rg\right)^{-1}g\cdot\partial_r,$$
which is certainly outwards--transverse to the boundary. Thus, the boundary is convex with respect to the stated Liouville structure.
\end{proof}
\noindent Given a contact structure and a choice of contact form, we have described the geometric properties of an open book decomposition supported by them. An open book decomposition supported by a contact structure will be shown to exist at the end of this Section. Part of the relevance of the open book decompositions in contact geometry also resides on the converse construction: we will able to obtain contact structures from the symplectic data associated to an open book allegedly supported by a contact form. To be precise, an open book decomposition $(P,\Psi)$ is said to be symplectic if $(P, d\beta)$ is an exact symplectic manifold with convex boundary and $\Psi\in\mbox{Symp}(P,\partial P;d\beta)$ is a symplectomorphism supported away from the boundary. Then we obtain the following
\begin{proposition}
Let $M=(P,\Psi)$ be a symplectic open book decomposition. Then, there exists a contact structure with contact forms supporting the open book decomposition. Further, any two such adapted contact forms that induce symplectomorphic (relative to the boundary) pages are isotopic through contact structures.
\end{proposition}

\begin{proof}
Let us first show existence. The contact structure will be constructed from a deformation of the constant distribution $\ker(\beta)$. Note that the case in which $\Psi$ is an exact symplectomorphism is particularly simple. Consider a smooth increasing cut--off function $c: [0,1] \longrightarrow [0,1]$ such that $c(t)|_{[0,0.1]}=1$ and $c(t)|_{[0.9,1]}=0$. Define on $P \times [0,1]$ the interpolating $1$--form
$$\beta_t= c(t)\Psi^*(\beta) +(1-c(t))\beta.$$
Then the form
\begin{equation}
\alpha_m= \beta_t +mdt \label{eq:mcont}
\end{equation}
is a contact form for $m$ large enough. Indeed, since
$$d\alpha_m=dt\wedge(\dot{c}(t)\Psi^*\beta+(1-\dot{c}(t))\beta)+d\beta$$
the contact condition reads
$$\alpha_m\wedge(d\alpha_m)^n = mdt\wedge (d\beta)^n+\eta$$
where $\eta$ is a $(2n+1)$--form independent of $m$. It remains to extend the form $\alpha_m$ to the relative suspension, that is to say, to fill the mapping torus $P_\Psi$. This will be done explicitly.\\

\noindent We use the characterization of the convex boundary of a symplectic manifold in terms of the symplectization, cf. \cite{Ge08}. Let $(M, d\alpha)$ be an exact symplectic manifold with convex boundary $B=\partial M$, then there exists a neighborhood $U$ of the boundary symplectomorphic to
$$(B \times (-\varepsilon, 0], d(e^s\cdot\alpha_{|B}),\quad s\in(-\epsilon,0].$$
In other words, a neighborhood is symplectomorphic to the symplectization of the contact manifold $(B, \alpha_{|B})$. In particular, the Liouville vector field reads $X= \partial_s$ in these coordinates.\\

\noindent Let us fill the mapping torus $P_\Psi$. A neighborhood $V$ of its boundary is of the form $V \stackrel{\varphi}{\cong} B \times (-\varepsilon, 0]\times S^1$. Fix coordinates $(p,s,t) \in B \times (-\varepsilon, 0]\times S^1$, then the contact form $\alpha_m$ in (\ref{eq:mcont}) is written as
$$
\varphi^* \alpha_m = e^s(\alpha_{|B})+ mdt.
$$
Defining the form in the filling is tantamount to an extension in a neighborhood of the boundary. Geometrically, we invert the model away from the section $B\times\{0\}\times S^1$ and glue it from the other side. In explicit terms, we consider the change of coordinates
\begin{eqnarray*}
\rho: B \times (-\varepsilon, 0)\times S^1 & \longrightarrow & B \times (0, \varepsilon)\times S^1\\
(p,s,t) & \longmapsto & (p,-s,t).
\end{eqnarray*}
In these coordinates our aim is to extend the form
$$\eta = e^{-s}\alpha_{|B}+ mdt= e^{-s}\cdot(\alpha_{|B}+ me^sdt)$$
to the gluing area $s=0$ preserving the contact condition. The contact structure will be defined on the whole open book since we may understand the $(s,t)$--coordinates as polar coordinates in the disk $B^2(\varepsilon)$. In the spirit of the proof of Lemma \ref{lem:special}, we define two smooth functions
$$H:[0,\varepsilon) \longrightarrow [0,1],\quad g:[0, \varepsilon)\longrightarrow [0,me^{\varepsilon}],$$
that contact interpolate between $\eta$ and the contact form in the boundary. Being precise, the functions must satisfy the following conditions:
\begin{itemize}
\item[-] $H|_{[\varepsilon/2,\varepsilon)}= me^s$ and $H|_{[0,\tau]}=s^2$ for an arbitrarily small $\tau<\epsilon$.\\For the contact condition, we require $\partial_s H>0$, for $s>0$.\\
\item[-] $g|_{[\varepsilon/2, \varepsilon)}=e^{-s}$ and $g|_{[0,\tau]}=1-s^2$ for an arbitrarily small $\tau < \epsilon$. After the lemma, since $s$ is the radial coordinate, $g$ should also satisfy $\partial_s g<0$, for $s>0$.
\end{itemize}
Finally we may construct the form
$$\widetilde{\eta}= g(s)(\alpha_{|B}+ H(s)dt),$$
coinciding with $\eta$ on the domain $B \times [\varepsilon/2, \varepsilon]\times S^1$ and extending the contact form $\alpha_m$ to a neighborhood of the boundary. The contact structure is adapted to the open book by construction.\\

Let us focus on the uniqueness statement. Consider two contact forms $\alpha_0$ and $\alpha_1$ adapted to the same open book $(B,\pi)$ and inducing the same symplectic structure on the leaves. Endow the manifold with a Riemmanian metric and define the function $d:M_B \to \R^+$ measuring the square of the distance from a point to the binding. This function is smooth close to the binding; smoothly deform $d$ so it becomes constant away from a small neighborhood of $B$, denote this deformation by $\widetilde{d}$. Consider the $1$--form $\nu= \widetilde{d} \cdot \pi^* (d\theta)$ defined over $M \setminus B$.  We construct the following deformation of any compatible contact form $\alpha$:
$$\alpha^t = \alpha + t \nu,$$
for $t\in[0,K]$ where the constant $K>0$ is arbitrarily large. Observe that the family $\alpha^t$ is a family of compatible contact forms. In order to connect the forms $\alpha_0$ and $\alpha_1$ we use the following linear family of compatible contact structures:
$\widetilde{\alpha}_t =(1-t)\alpha^K_0 +t \alpha^K_1.$
By Gray's stability we conclude the uniqueness of the contact structure.
\end{proof}

As previously mentioned, we will explain a converse of this result. There are two different cases depending on the dimension of the manifold being $3$ or higher. In dimension $3$ there is a strong statement that ensures a complete equivalence:
\begin{theorem} $($Giroux$)$
Let $M$ be a smooth manifold. There exists a one--to--one correspondence between contact structures over $M$ up to isotopy and symplectic open book pairs $(P, \phi)$ associated to $M$ up to positive stabilization.
\end{theorem}
\noindent For the notion of stabilization and an account of the proof of this result, see \cite{Co08}.
The higher--dimensional analogue is weaker. The statement is:
\begin{theorem} $($Giroux$)$ \label{thm:exist_open}
Let $(M, \xi)$ be a contact manifold. There exists a contact form $\alpha$ for the contact structure $\xi$ supporting an open book.
\end{theorem}
\begin{proof} The proof constructs the open book decomposition using the theory of asymptotically holomorphic sections. Let us divide the argument in $3$ parts: construction of the binding as a contact divisor, obtaining the topological fibration over the circle and description of the contact form following Lemma \ref{lem:special}.\\

\noindent {\it Step 1: The binding}\\
We first describe the input data require to use the methods of Section \ref{section:asymp}: we select a contact form $\alpha$ for the contact distribution $\xi$, fix a compatible almost complex structure $J$, construct the prequantizable bundle $L=M \times \C$ with associated connection $\nabla= d -i\alpha$ and the sequence of metrics $g_k$ that we will compute the norms with. Theorem \ref{thm:Ber} provides us with an asymptotically holomorphic sequence of sections $s_k: M \to L^{\otimes k}$ which are $\varepsilon$--transverse to zero along $\xi$. For $k$ large enough, the zero locus $B_k=Z(s_k)$ is a contact submanifold with trivial normal bundle. This is the contact divisor that will be used as binding.\\

\noindent {\it Step 2: The topological open book} \newline
Consider the following sequence of maps
\begin{eqnarray*}
\pi_k: M\setminus B_k & \longrightarrow & S^1, \\
p & \longmapsto & \frac{s_k(p)}{|s_k(p)|}.
\end{eqnarray*}
The section $s_k:M \to L^{\otimes k}= \C$ is being understood as a $\C$--valued smooth function since $L$ is topologically trivial. Consider the following sequence of open covers $M=U_k\cup V_k$, where
$$U_k= \left\{ p \in M: |s_k| < \frac{\varepsilon}{2}\right\}\quad\mbox{and}\quad V_k= \left\{ p \in M: |s_k| > \frac{\varepsilon}{4}\right\}.$$
The covariant derivative reads as
$$\nabla s_k(p) = ds_k(p) -ik\alpha s_k(p),$$
and thus deriving in the Reeb vector field direction we obtain
$$\nabla s_k(p) = d_R s_k(p) -iks_k (p).$$
Since the sections satisfy the asymptotically holomorphic bounds $|\nabla s_k|=O(1)$ and $|R|_k= k^{1/2}$, the $g_k$--norm of the derivative in the Reeb direction can be estimated as
$$| d_Rs_k(p) +iks_k(p) |=O(k^{1/2}).$$
Hence $d_R s_k(p) \equiv -iks_k(p)$ in the open set $V_k$. A brief computation shows that $\pi_k$ is a submersion in this situation and that the Reeb vector field $R$ is transverse to the fibers.\\

\noindent To conclude analogously for $U_k$ we use the directions in the distribution. For any vector $e_p \in \xi_p$ the covariant derivative reads $\nabla_{e_p} s_k= d_{e_p} s_k$, therefore the $\varepsilon$--transversality ensures that for any point $p$ in the region $U_k$, there are two vectors $u_k, v_k \in \xi_p$ such that the map $\nabla s_k: \xi_p \longrightarrow \C$ is surjective restricted to them. Consequently so is the map $d s_k: \xi_p \longrightarrow \C$ when restricted to them, consequently the map $\pi_k$ is a submersion on $U_k$. The implicit function theorem provides the topological local model for the function close to the binding.\\

\noindent {\it Step 3: Contact form close to the binding.} \newline
The contact form in a neighborhood over the binding will be constructed as a contact fibration over the normal disk. We should ensure that the fibers close to the binding are also contact submanifolds, i.e. the sets $B_k(t)= Z(s_k-t)$ for $t\in B^2(\varepsilon/2)\subset\C$ are contact submanifolds. It is important to notice that the sequence $s_k-t$ is no longer asymptotically holomorphic, since the sequence of sections $\sigma_k=t$ is not asymptotically holomorphic for the derivatives in the Reeb direction do not satisfy the asymptotically holomorphic bounds. Since $s_k-t$ are $\frac{\varepsilon}2$--transverse to zero along $\xi$, the sets $B_k(t)$ are at least smooth submanifolds. However, the antiholomorphic part is bounded as $|\bar{\partial} (s_k -t)| =O(k^{-1/2})$ and therefore the sets $B_k(t)$ are also contact submanifolds. In particular, the projection map
\begin{eqnarray*}
\Pi_k: V_k & \longrightarrow & D^2(\varepsilon) \\
p & \longmapsto & s_k(p)
\end{eqnarray*}
is a contact fibration. We now use the contact fibration methods from the proof of Lemma \ref{lem:special} to obtain a positive function $\delta: B\longrightarrow \R^+$ such that the domain $\widetilde{V}_k= \{ (p, v) \in B \times \R^2: |v| \leq \delta(p) \}$ admits a diffeomorphism
$$\phi_k: B_k \times D^2(\delta)\longrightarrow\widetilde{V}_k, $$
such that $\phi_k^*(\alpha)= \widetilde{g}\cdot(\alpha_B +r^2 d\theta)$, for a positive function $\widetilde{g}$. This diffeomorphism is also compatible with the circle projection, i.e. if
$$\pi_\theta: \widetilde{V}_k \setminus B \times \{ 0 \}\longrightarrow S^1$$
denotes the projection into the angular coordinates, then $\pi_\theta= \pi_k \circ \phi_k$. Note that there is no a priori guarantee that $\partial_r \widetilde{g}<0$ for $r>0$, which is the condition for the contact form to be adapted. However, we have already verified that $\alpha$ supports the open book in the neighborhood of the boundary $\phi_k^{-1}(U_k \cap \widetilde{V}_k)$, in particular $\partial_r \widetilde{g}<0$ in this region. Thus, it remains to find a function $g:\widetilde{V}_k\longrightarrow\R^+$ such that:
\begin{enumerate}
\item[-] $g$ extends $\widetilde{g}$, i.e. $g=\widetilde{g}$ over $\phi_k^{-1}(U_k \cap\widetilde{V}_k)$.
\item[-] It satisfies the contact condition $\partial_r g<0$ for $r>0$ and for small radii $0\leq r<r_0$ it coincides with the local model $g(p,r)=k-r^2$ for some large $k>0$.
\end{enumerate}
The function $g$ always exists. Finally, we define the contact form
$$ \widetilde{\alpha}= g\cdot(\alpha_B +r^2 d\theta),$$
extending $\phi_k^*(\alpha)$ beyond $\phi_k^{-1}(U_k \cap\widetilde{V}_k)$. This form can be extended through $\phi_k$ to an adapted global form on the open book decomposition $M=(B_k, \pi_k)$.
\end{proof}
\begin{remark}
We believe that is possible to construct the adapted contact form to be $C^0$--close to the arbitrary initial contact form. This would require the use of the asymptotically holomorphic theory as developed in \cite{IMP99}, since we need a better control for the derivatives in the Reeb direction to ensure the bound
$$| \nabla_R s_k(p)| =O(k^{s}),$$
where $0\leq s < 1/2$. This would imply that the Reeb vector field would be asymptotically tangent to the submanifold $B_k$ and therefore, for $k$ large enough, we could obtain $|\partial_r\widetilde{g}|<\gamma$, for $\gamma>0$ arbitrarily small.

\end{remark}
\noindent As mentioned in the introduction this existence result was essentially proved by E. Giroux almost 10 years ago. E. Giroux and J. P. Mohsen are writing a monography \cite{GM12} containing this result and a more complete dictionary relating open books and contact structures. We briefly cite some of the main results in the area:
\begin{itemize}
\item[-] Uniqueness up to stabilization of the open books constructed in the preceding Theorem \ref{thm:exist_open}.
\item[-] Equivalence between convex decompositions, as defined in \cite{Gi91}, and adapted open book decompositions.
\item[-] In the case of a Stein fillable contact manifold, the open book can always be understood as the boundary of a Lefschetz pencil over the disk. As a corollary, they obtain that a contact manifold is Stein fillable if and only if it admits an open book whose monodromy map is generated by positive Dehn--Seidel twists.
\item[-] Relations with the existence of contact structures in higher dimensions. In particular, existence of contact fibrations.
\end{itemize}
%-----------------------------------------------
\section{Pencils in quasi--contact and contact geometry}\label{sect:pencil}

In this section we explain the second type of decomposition mentioned in Section \ref{sect:introduction}: the analogue of the Lefschetz pencil in a symplectic manifold. They were initially introduced in \cite{Pr02} for the contact case and in \cite{Ma09} for the quasi--contact one. The geometric construction still consists in projecting the manifold to reduce the dimension. In this case we will produce a projection onto $\CP^1$, thus the fibers become real codimension--$2$ submanifolds.\\

\subsection{Definitions.} Let $(M, \xi, d\beta)$ be a quasi--contact structure and fix a compatible almost--complex structure for the symplectic bundle $(\xi,d\beta)$. A chart $\phi:(U,p)  \longrightarrow V \subset (\C^n \times \R, 0)$ is said to be compatible with the quasi--contact structure at a point $p \in U \subset M$ if the push--forward at $p$ of $\xi_p$ is the hyperplane $\C^n \times \{ 0\}$ and $\phi_*d\beta_p$ is a positive $(1,1)$--form. The central notion in this section is the content of the following:
\begin{definition} \label{def:almost_pencil}
A quasi--contact pencil on a closed quasi--contact manifold $(M^{2n+1}, \xi,
d\beta)$ is a triple $(f,B,C)$ consisting of a codimension--$4$ quasi--contact
submanifold $B$, called the base locus, a finite set $C$ of smooth
transverse curves and a map $f:M\backslash B\longrightarrow \C\mathbb{P}^1$
conforming the following conditions:
\begin{itemize}
\item[(1)] The set $f(C)$ contains locally smooth curves with transverse
self--intersections and the map $f$ is a submersion on the complement of $C$.
\item[(2)] Each $p\in B$ has a compatible local coordinate map to $(\C^n \times
\R,0)$ under which $B$ is locally cut out by $\{z_1=z_2=0\}$ and $f$ corresponds
to the projectivization of the first two coordinates, i.e. locally $\displaystyle f(z_1, \ldots, z_n,
t)= \frac{z_2}{z_1}$. 
\item[(3)] At a critical point $p\in \gamma\subset M$ there exists a compatible
local coordinate chart $\phi_P$ such that $$(f\circ \phi_P^{-1})
(z_1,\ldots,z_n,s)=f(p)+z_1^2+\ldots+z_n^2+g(s)$$ where $g:(\R,0)\longrightarrow (\C,0)$ is
a submersion at the origin.
\item[(4)] The fibers $f^{-1}(P)$, for any $P\in \CP^1$, are quasi--contact
submanifolds at the regular points.
\end{itemize}
\end{definition}

\begin{figure}[ht]
\includegraphics[scale=0.55]{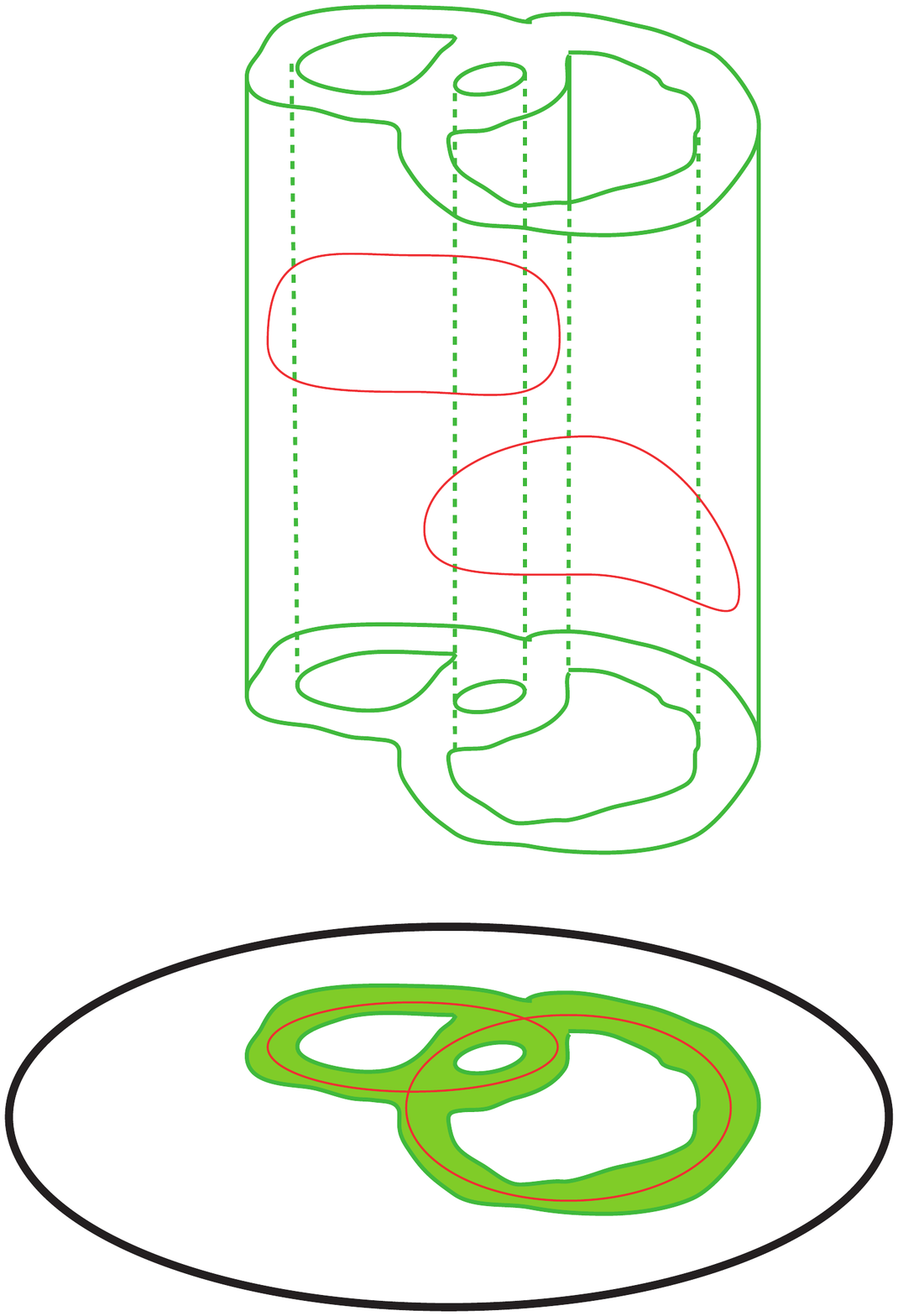}
\caption{Counter--image of a neighborhood of two curves of critical values.}
\end{figure}

\noindent These objects always exist on a quasi--contact manifold. Actually, it is even possible to partially prescribe the topology of the fibres:
\begin{theorem} \label{thm:exist_pencil}
Let $(M, \xi, d\beta)$ be a quasi--contact manifold. Given an integral class $a\in
H^2(M, \Z)$, there exists a quasi--contact pencil $(f, B, C)$ such that the
fibers are Poincar\' e dual to the class $a$.
\end{theorem}

\noindent This existence result can be readily extended to a more general notion of quasi--contact structures, i.e. triples of objects $(M, \xi, \omega)$, with $\xi$ a codimension--$1$ cooriented distribution and $\omega$ a closed $2$--form of integral class such that $(\xi,\omega)$ is a symplectic bundle. In \cite{Ma09} the theory is developed for these objects, though no further applications have been found in that more general setting.\\

The remaining part of this section is dedicated to the proof of Theorem \ref{thm:exist_pencil}. The strategy mimics the construction of Lefschetz pencils in projective geometry: we will produce a pair $(s_0,s_1)$ of suitable sections of a complex line bundle, thought of as a basis for a $1$--dimensional linear system, and use them to map the quasi--contact manifold onto $\CP^1$. As in Section \ref{sect:ob}, the asymptotically holomorphic theory from Section \ref{section:asymp} will provide the sections. Observe that in this occasion we will produce a pair of sections and there should be further control for the behaviour of the their quotient.\\

\noindent The initial data to obtain the sections is a quasi--contact form $\beta$, a compatible almost complex structure $J$ for the symplectic bundle $(\xi, d\beta)$, a sequence of metrics $g_k$ and the prequantizable bundle $L=M \times \C$ associated to the connection $\nabla= d - i\beta$. In order to prescribe the Poincar\'e dual of the fibers, let $V$ be a fixed hermitian line bundle with a connection such that the associated curvature $\Theta_V$ satisfies $[\Theta_V] = a$. Then the existence theorems from Section \ref{section:asymp} allow us to prove the following
\begin{proposition} \label{propo:nice}
With the data as described above, there exists an asymptotically holomorphic sequence of sections
$$s_k=(s_{k,0}\oplus s_{k,1}): M \longrightarrow V \otimes L^{\otimes k} \otimes \C^2$$ and two fixed constants $\varepsilon, \varepsilon'>0$ satisfying that
\begin{enumerate}
\item[-] $s_k$ is $\varepsilon$--transverse to zero along $\xi$ over $M$.
\item[-] $s_{k,0}$ is $\varepsilon$--transverse to zero along $\xi$ over $M$,
\item[-] Consider the set $W_{\infty}^{s_k}= \{ p\in M: s_{k,0}(p)=0\}$. Then the holomorphic part $\partial\left(\frac{s_{k,1}}{s_{k,0}}\right)$ of the covariant derivative is $\varepsilon'$--transverse to zero along $\xi$ over $M \setminus W_{\infty}^{s_k}$.
\end{enumerate}
\end{proposition}
\begin{proof}
The $\varepsilon$--transversality is a $C^1$--stable property, thus we may systematically perform $C^1$--perturbations and it will be preserved. Consider the asymptotically holomorphic null--constant sequence of sections $(0): M \longrightarrow L^{\otimes k} \otimes \C^2$, then Theorem \ref{thm:Ber} provides an asymptotically holomorphic sequence of sections $s_k$ which are $\varepsilon_1$--transverse to zero.\\

\noindent Let $\delta= \varepsilon_1/2$ and apply Theorem \ref{thm:Ber} to the sequence $s_{k,0}:M\longrightarrow L^{\otimes k}$ in order to obtain an asymptotically holomorphic sequence of sections $\sigma_{k,0}: M \longrightarrow L^{\otimes k}$ such that $|s_{k,0}-\sigma_{k,0}|_{C^2} \leq \delta$ and $\varepsilon_2$--transverse to zero along $\xi$. The pair $\tilde{s}_k=(\tilde{s}_{k,0},\tilde{s}_{k,1})=(\sigma_{k,0}, s_{k,1}): M \longrightarrow L^{\otimes k} \otimes \C^2$ satisfies the first two transversality properties for $\varepsilon_3= min(\delta, \varepsilon_2)$.\\

\noindent Observe that $\partial \left(\frac{\tilde{s}_{k,1}}{\tilde{s}_{k,0}} \right)$ is an asymptotically holomorphic sequence on the open set $U_{\varepsilon}^{s_{k}}=\{p\in M: |s_{k,0}| > \varepsilon \}$ since it is
$$ \partial \tilde{s}_{k,1} \otimes \tilde{s}_{k,0} - \tilde{s}_{k,1} \otimes \partial  \tilde{s}_{k,0} : M \longrightarrow \xi^{1,0}\otimes L^{\otimes k} \otimes V \otimes \C^2. $$
Finally, apply Theorem \ref{thm:Ber2} to the section $\tilde{s}_k$ and the constant $\varepsilon_3/2$: there exists an asymptotically holomorphic sequence $\tilde{\sigma}_k$ such that  $\partial \tilde{\s}_{k,1} \otimes \tilde{\s}_{k,0} - \tilde{\s}_{k,1} \otimes \partial  \tilde{\s}_{k,0}$  is $\varepsilon_4$--transverse to zero along $\xi$ over $M$. This sequence is satisfies that $\partial \left(\frac{\tilde{\s}_{k,1}}{\tilde{\s}_{k,0}} \right)$ is $\varepsilon'$--transverse to zero along $U_{\varepsilon_3/2}^{\s_k}$. Thus, the sequence $\tilde{\s}_k$ satisfies the required properties for the chosen $\varepsilon'$ and $\varepsilon=\varepsilon_3/2$.
\end{proof}

\subsection{Existence of quasi--contact pencils.}
Let us briefly describe the argument. Apply Proposition \ref{propo:nice} to the data induced by the quasi--contact manifold $(M,\xi,d\beta)$ and $a\in H^2(M,\Z)$ as previously explained. This provides a pair of suitably transverse sections inducing the potential quasi--contact pencil. The first step is the structure of the fibers, which should satisfy (4) in Definition \ref{def:almost_pencil}. Secondly we focus on the base locus and obtain the required local model. Finally, it is ensured that the Morse model around the singularities can be achieved.\\

\noindent {\it Step 1: Analysis.} Since the sections $s_k$ provided by Proposition \ref{propo:nice} are $\varepsilon$--transverse, the zero set $B_k=Z(s_k)$ is a codimension $4$ quasi--contact manifold. Also, the set $W_{\infty}^{s_k}$ is a codimension $2$ quasi--contact submanifold. Let us define the sequence of maps
\begin{eqnarray*}
F_k: M \setminus W_{\infty}^{s_k} &\longrightarrow & \C \\
p & \longmapsto & \frac{s_{k,1}}{s_{k,0}}.\end{eqnarray*}
These are our candidates for quasi--contact pencil structures. Define the set
$$\Gamma=\{ p\in M: |\partial F_k|\leq |\bar{\partial} F_k| \}.$$
If we are able to show that $\Gamma= \{p\in M: d_\xi F_k=0 \}$ then the fibers of $F_k$ will be quasi--contact submanifolds at the regular points. We will actually justify that the set $\Gamma$ lies arbitrarily close to the critical curves. First, a bound from below for the norm of $s_{k,0}$ on $\Gamma$:
\begin{lemma} \label{far}
There is a constant $\eta>0$, depending only on $\varepsilon'$ and $\varepsilon$, such that if $k$ is
large enough, then $|s_{k,0}|\geq \eta$ on $\Gamma$.
\end{lemma}
\begin{proof}
The section $s_k$ is $\varepsilon$--transverse to zero along $\xi$ and thus at any point $p \in M$ with $|s_k(p)| <\varepsilon$ the map $\partial s_k(p)$ is $\varepsilon$--transverse. Without loss of generality suppose that $|s_{k,0}| \leq |s_{k,1}|$. By surjectivity there exists a unitary vector $v\in \xi_p$ such that $|\partial_v s_{k,0}(p)|\geq \varepsilon$ and $|\partial_v s_{k,1}(p)|=0$ and thus
\begin{equation}
|\partial s_{k,1}(p) \otimes s_{k,0}(p) - s_{k,1}(p) \otimes \partial  s_{k,0}(p)|\geq \frac{|s_k(p)|}{2} \varepsilon\label{eq:holo1}
\end{equation}
The asymptotically holomorphic bounds impose
\begin{equation}
|\bar\partial s_{k,1}(p) \otimes s_{k,0}(p) - s_{k,1}(p) \otimes \bar\partial  s_{k,0}(p)|\leq ck^{-1/2} |s_k(p)|\label{eq:antiholo1}
\end{equation}
and so combining the inequalities (\ref{eq:holo1}) and (\ref{eq:antiholo1}), for $k$ large enough, we obtain
$$
|\partial s_{k,1}(p) \otimes s_{k,0}(p) - s_{k,1}(p) \otimes \partial  s_{k,0}(p)|> |\bar\partial s_{k,1}(p) \otimes s_{k,0}(p) - s_{k,1}(p) \otimes \bar\partial  s_{k,0}(p)|. $$
This implies $|\partial F_k(p) | > |\bar \partial F_k(p) |$ at any point $p\in M\setminus W^{s_k}_{\infty}$ with $|s_k(p)| <\varepsilon$.\\

\noindent Consider $p\in M$ with $|s_k(p)|>\varepsilon$, say $|s_{k,1}(p)|\geq\varepsilon/2$. Let $\eta\leq\varepsilon$ and suppose further that $|s_{k,0}(p)| < \eta$, by $\varepsilon$--transversality of $s_{k,0}$ the inequality $|\partial s_{k,0}(p)| > \varepsilon$ holds. At the same time the asymptotically holomorphic bounds require $|\partial s_{k,1}(p)| \leq c$, for some fixed constant $c>0$. Fix $\eta=\varepsilon^2/4c$, then the reverse triangle inequality yields
$$|\partial s_{k,1}(p) \otimes s_{k,0}(p) - s_{k,1}(p) \otimes \partial  s_{k,0}(p)|\geq \frac{\varepsilon^2}{4}.$$
Again (\ref{eq:antiholo1}), for $k$ large enough, gives $|\partial F_k(p) | > |\bar \partial F_k(p) |$.
\end{proof}

Let $\Delta\subset\Gamma$ be the set of points where $\partial F_k=0$. The connected components of
$\Delta$ form a discrete set of smooth transverse curves since $\partial F_k$ satisfies the adequate transversality condition. Observe that $\pi_0(\Delta)$ is finite because $\Delta\subset\Gamma$ and the set $\Gamma$ is contained in the complementary
of a $\tau$--neighborhood of the compact manifold $W_{\infty}=Z(s_k^0)$, for $\tau>0$ a constant
small enough, after Lemma \ref{far}. In order to understand the behaviour of the set $\Gamma$, consider the set $\Omega_{\eta}= \{ p\in C, |s_0(p)|> \eta/2 \}$. The following statement describes the neighborhoods of the elements in $\Delta$ and in particular the geometry of $\Gamma$:
\begin{proposition} \label{disjoint}
With the above notation, there exists a uniform constant $\rho_0>0$ such that the $\rho_0$--neighborhoods
of each connected component $\gamma_i\in\Delta$ are disjoint and contained
in $\Omega_{\eta}$. Further, given any $\rho<\rho_0$, for $k=k(\rho)$ large enough, the set $\Gamma$ is contained in a $\rho$--neighborhood of $\Delta$.
\end{proposition}
This is essentially Proposition 9 in \cite{Do99}, instead of the distance to a finite number of points we use the distance from a point to a curve. Geometrically, in a point of $\Gamma$ the norm $|\partial F_k|$ is bounded by $|\overline{\partial} F_k|=O(k^{-1/2})$ and thus can be arbitrarily small, transversality then provides a solution for the equation $\partial F_k=0$ and the norm being arbitrarily small ensures its existence nearby. The detailed argument requires an explicit form of the Inverse Function Theorem, cf. \cite{Do99}.\\

\noindent {\it Step 2: Perturbation at the base point set.}
The sequence $\partial F_k$ will be perturbed in arbitrarily small neighborhoods
of the alleged base locus $B_k$ in order to achieve the local model in Definition \ref{def:almost_pencil}. To ease notation we write $B=B_k$, as $k$ is thought as fixed if large enough. We describe the local model in terms of the equation for the tangent space at a point $p\in B$
$$ \nabla s_k(p)= \nabla s_0^k(p) \oplus \nabla s_1^k(p): T_pM\longrightarrow L_p^{\otimes k} \oplus
L_p^{\otimes k}.$$
If $T_pB\subset\xi_p$ were a symplectic subspace, the $\R$--linear map $\nabla s_k(p)$ provides an isomorphism as $\R$--vector spaces of the symplectic orthogonal with the $\C$--vector space $L_p^{\otimes k} \oplus
L_p^{\otimes k}$. In particular the symplectic orthogonal is endowed with a complex structure. Let us prove a linear characterization of the model with respect to the base locus:
\begin{lemma} \label{positi}
Let $p\in B$. Then $F_k$ can be represented around $p$ in the standard local model of Definition
\ref{def:almost_pencil} if and only if $T_pB\cap (\xi_p,d\beta_p)$ is a symplectic subspace and the
restriction of $d\beta$ to the symplectic orthogonal $N_p=(T_pB\cap \xi)_x^{\perp d\beta}$ is a
positive form of type $(1,1)$ with respect to the complex structure of $N_p$ induced by $\nabla s_k(p)$.
\end{lemma}
\begin{proof}
It is readily seen that the condition is necessary since the existence of a local model provided in the Definition \ref{def:almost_pencil} implies the properties in the statement. Conversely, suppose that $B$ is a quasi--contact submanifold near $p$ and let $(z_3, \ldots, z_n,t)$ be local coordinates at $p$ such that
$$\displaystyle(d\beta)_{|B} = \frac{i}{2} \sum_{j=3}^n dz_j \wedge d\bar{z}_j\quad\mbox{ and }\quad\xi_p= \ker dt.$$
Extend the coordinate functions $(z_3,\ldots,z_n)$ ensuring that the their derivatives vanish in the normal directions of $N_p$. Locally trivialize the bundle $L^{\otimes k}$ via a non--vanishing section $\sigma$ and define functions $z_0= s_{k,0} \cdot \sigma$, $z_1= s_{k,1} \cdot \sigma$. These provide a complete set of coordinates $(z_1,\ldots, z_n,t)$ around $p$ in which the symplectic form is expressed as
$$(d\beta)_{|B} = (d\beta)_{|N_p} + \sum_{j=3}^n dz_j \wedge d\bar{z}_j.$$
Thus we obtain the required local model.\end{proof}

To achieve the local model at $B$, it remains to perturb $F_k$ such that it satisfies the hypothesis of the linear characterization. More precisely, there exists a perturbation $D_p: T_pM \longrightarrow L_p^{\otimes k}\oplus L^{\otimes k}$ of the map $\nabla s_k(p)$ conforming
$$\left|\left(\nabla s_k(p)-D_p\right)v\right| \leq ck^{-1/2}|\nabla s_k(p)(v)|,\quad \forall v\in T_pM,$$
and the requirements of the Lemma \ref{positi}. Indeed, it is simple to perturb the pair $s_k=s_{k,0}\oplus s_{k,1}$ to $\tilde{s}_k=\tilde{s}_{k,0}\otimes \tilde{s}_{k,1}$ at distance $O(k^{-1/2})$ in $C^3$--norm with $D_p$ as the linearization at $p$ of associated pencil--map $\tilde{F}_k$. This perturbation fulfills the property $(2)$ of the Definition \ref{def:almost_pencil}. The perturbation will still be referred as $F_k$.\\

\noindent {\it Step 3: Local model for the singularities.} It remains to study the map $F_k$ near the singular set $\Delta$. Let $\gamma\in\Delta$ be a smooth connected curve. The perturbation of $F_k$ will occurs in a $\delta$--neighborhood of $\Delta$, Proposition \ref{disjoint} implies that the perturbations can be independently performed in each connected component of $\Delta$. We will describe the associated quasi--contact data around the curve $\gamma$, define a general perturbation well--behaved with respect to a simple integral distribution and then prove that it can be chosen to induced with Morse model with respect to the actual quasi--contact distribution.\\

Let us specify the information contained in a trivialization. For $k$ large enough, the curve $\gamma$ is a transverse contact loop, equivalently $TM|_\gamma=T\gamma\oplus\xi$, and the angle between $T\gamma$ and
$\xi$ is bounded below by a uniform constant because of the transversality of the sequence.
To trivialize we use the geodesic flow of the metric $g_J$ associated to the fixed almost complex structure $J$ and obtain a diffeomorphism
$$\phi: U_{\rho}\longrightarrow V_{\rho} \subset S^1\times \C^n, $$
where $U_{\rho}$ is a $\rho$--neighborhood of $\gamma$, measured with the $g_k$ metric, and $V_{\rho}$ its image by the flow, which is an open neighborhood of $S^1\times \{ 0 \}$. In the neighborhood $S^1\times\C^n$ we consider the product metric with first component the image of $g_k$ through $\phi$ and second component the standard hermitian metric in $\C^n$. Suppose also that $\phi_* J_{|\gamma}=J_0$, $J_0$ being the standard complex structre of $\C^n$. In particular, the model being isometric allows us to explicitly measure in $V_\rho$.\\

\noindent Regarding the distributions, denote $\xi_k=\phi_*\xi$ and let $\xi_h$ be the integrable distribution given as $\{ p \}
\times \C^n\subset S^1\times \C^n$. Possibly after a uniform shrinking of $\rho$, the angle\footnote{The maximum angle between two subspaces $U,V \subset \R^m$ of the Euclidean space is by definition $\angle_M(U,V) = \max_{u\in U} \{\angle(u,V)\}.$} between $\xi_k$ and $\xi_h$ tends to zero; more precisely, there exists a uniform constant $C>0$ such that
\begin{equation}
\angle_M(\xi_k(s,z), \xi_h(s,z))< C|z|k^{-1/2}, ~ \forall (s,z)\in V_{\rho}
\subset S^1\times \C^n, \label{apr_dist}
\end{equation}
Hence, we are able to project orthogonally the almost complex structure $\phi_*J$ on $\xi_k$ to an almost complex structure $J_h$ in the distribution $\xi_h$. Let $\mu:\bigwedge^{(1,0)}_{J_0}\longrightarrow\bigwedge^{(0,1)}_{J_0}$ be the defining function of $J_h$ with respecto to $J_0$ as almost complex structures in $\xi_h$, cf. \cite{Do96}. Denote by $\partial$ and $\partial_0$ the associated holomorphic parts of the covariant derivative with respect to $J_h$ and $J_0$. The holomorphic and antiholomorphic parts of the operator corresponding to $\phi_*J$ in $\xi_k$ will be denoted $\partial_k$ and $\bar\partial_k$. In particular
\begin{equation}
\partial=\partial_0+\bar{\mu}\bar{\partial}_0, ~~ \bar{\partial}=\bar{\partial}_0+
\mu\partial_0. \label{apr_cplx}
\end{equation}
It follows that $|\mu(z)|\leq C|z|k^{-1/2}$, $C$ being a uniform
constant.\\

Let us define the perturbation of $F_k$, as mentioned above we will construct the perturbation in the isometric neighborhood $S^1\times\C^n$. The context being clear, we will still denote by $F_k$ the pull--back $(\phi^{-1})^*F_k$. Since the local model required from \ref{def:almost_pencil} is quadratic, we will deform $F_k$ to approximately its complex Hessian $H=\frac{1}{2}\nabla_\partial(\partial F_k)$. Locally, it is expressed as
$$ H(s,z)=\sum H_{\alpha \beta}(s) z_{\alpha}z_{\beta}.$$
Consider a cut--off function $\beta_{\rho}: S^1\times\C^n\longrightarrow[0,1]$ satisfying
\begin{enumerate}
\item[-] $\beta_{\rho}(\phi(p))=1$ if $d_k(p,\gamma)\leq\rho/2$, and $\beta_{\rho}(\phi_(p))=0$ if $d_k(p, \gamma)\geq \rho$.
\item[-] $|\nabla \beta_{\rho}|=O(\rho^{-1})$.
\end{enumerate}
The second condition can be ensured due to the choice of metrics. The constant $\rho<\rho_0$ will be shrunk in a uniform way, and so we consider it fixed assuring that the conditions are satisfied. The perturbation of $F_k$ will be of the form
$$\widetilde{F}_k(s,z)=\beta_{\rho}(w(s)+H(s,z))+(1-\beta_{\rho})F_k(s,z), $$
where $w:S^1\longrightarrow\C$ is any smooth function. The only further issue is the verification of $\Gamma=\Delta$ for the perturbed $\widetilde{F}_k(s,z)$, this is the content of the following\\

\begin{lemma}
With the above notations, let $\rho>0$ and $|w(s)-F_k(s,0)|$ be small enough and $|\dot{w}(s)|=O(1)$. Then for sufficiently large $k$ the inequality
$|\partial_k\widetilde{F}_k|\leq |\bar{\partial}_k \widetilde{F}_k|$ is only satisfied in $\gamma$.
\end{lemma}
\begin{proof}
There are two different scenarios, close to the curve where $\beta_{\rho}\equiv1$ and the transition area where $\nabla\beta_p$ does not vanish. Let us first consider the former. Then the perturbation reads $\widetilde{F}_k=w+H$ and so $ \partial \widetilde{F}_k=\partial H$, $\bar{\partial}\widetilde{F}_k=\bar{\partial} H$. The $\eta$--transversality of $\partial F_k$ yields the following bound
$$ |\partial H(s,z)| \geq \eta |z| - |\bar{\partial}(\partial F)_{(z=0)}|z|. $$
Since $\bar{\partial} \partial+ \partial \bar{\partial}\equiv0$ on functions
and the norm $\partial \bar{\partial} F$ is controlled, as $F$ has uniform
$C^3$--bounds, we obtain
$$ |\partial H(s,z)| \geq \eta |z| - Ck^{-1/2}|z|.$$
We need to relate this to the distribution $\xi_k$. By hypothesis $|\dot{w}(s)|=O(1)$ and we also know $|\partial s H_{\alpha \beta}(s)|=O(1)$ due to the $C^3$--bounds of $F_k$. Then the angle inequality (\ref{apr_dist}) implies
\begin{equation}
|\partial_k H(s,z)| \geq \eta |z| - Ck^{-1/2}|z|, \label{upper}
\end{equation}
where $C>0$ is another suitable uniform constant; the uniform constants appearing on the bounds will deliberately be referred as $C$. The asymptotically holomorphic bounds also imply that $|\bar{\partial} H| \leq C|z|^2k^{-1/2}$ and we analogously deduce
\begin{equation}
|\bar{\partial}_k H| \leq C(|z|^2k^{-1/2}+|z|k^{-1/2}). \label{lower}
\end{equation}
The condition $|\partial_k H|\leq |\bar{\partial}_k H|$ along with (\ref{upper}) and (\ref{lower}) implies $z=0$ for $k$ large enough, concluding the statement in this case.\\

\noindent We focus on the latter situation, i.e. the behaviour of $\Gamma$ around points in the annulus containing the support of $\nabla \beta_{\rho}$. The antiholomorphic derivative of the perturbation reads
$$\bar{\partial} \widetilde{F}_k =\bar{\partial} \beta_{\rho}(w+H-F_k)+\beta_{\rho}\bar{\partial}
H+(1-\beta_{\rho})\bar{\partial} F_k. $$
As before this concerns $\xi_h$ and we may bound the norm $|\bar{\partial} f_0|$ as in \cite{Do99} . Again the hypothesis $|\dot{w}(s)|=O(1)$ and the asymptotically holomorphic estimate $|\bar{\partial}_k F_k|=O(k^{-1/2})$ combine with the angle inequality (\ref{apr_dist}) to conclude
$$ |\bar{\partial}_k \widetilde{F}_k| \leq C(\rho^2+k^{-1/2}+|\widetilde{F}_k(s,0)-w|\rho^{-1}). $$
The direct computation $\partial_k \widetilde{F}_k =\partial_k \beta_{\rho}(w+H-F_k)+\beta_{\rho}\partial_k H
+(1-\beta_{\rho})\partial_k F_k$ and the transversality of $F_k$ yield a lower bound for $|\partial_k \widetilde{F}_k|$. The
argument follows as in \cite{Do99} until
$$ |\partial_k \widetilde{F}_k|-|\bar{\partial}_k \widetilde{F}_k| \geq \frac{\eta \rho}{2}-C(\rho^2+k^{-1/2}
+|w-F_k(s,0)|\rho^{-1}).$$
By the hypothesis $|w-F_k(s,0)|$ is small enough, and once fixed a sufficiently small $\rho$, for $k$ large enough the inequality is strictly positive over the annulus. This concludes the statement in the second case.
\end{proof}
Note that $w$ can be chosen generic enough to ensure that the projection of the critical points is a family of immersed curves. Also, the perturbation satisfies the local model around the curves because a real generic $S^1$--family of non--degenerate quadratic forms can be diagonalised. This proves the existence of quasi--contact pencils. The statement concerning the Poincar\'e dual of the fibers follows from the fact that first Chern class of the normal bundle to the section is the Poincar\'e dual of its vanishing locus.

%
%\section{From pencils to open books.}

%As a cross check of this fact, realize that almost contact structures in $M$ are sections of the bundle $Fr(TM)/ (U(2) \times \Z_2)$, where $Fr(X)$ is the principal $O(5)$-bundle of frames of the tangent bundle. The quotient is produced by the inclusion of Lie groups
%\begin{eqnarray*}
%U(2)\times \Z_{2} & \longrightarrow & \hspace{0.5cm}O(5) \\
%(u,a) & \longrightarrow & \left( \begin{array}{cc} a\cdot u  & 0 \\ 0 & a \end{array} \right).
%\end{eqnarray*}
%It is simple to check that the fiber is diffeomorphic to $SO(5) / U(2)$. In the other hand the $U(2)$-bundles inside $TM\otimes L$ are parametrized by the sections of the bundle $Fr(TM\times L)/U(2)$ whose fiber is $SO(5) / U(2)$. Therefore the spaces of sections have the same obstruction theory and this is a reason for them to be in bijective correspondence. \\
\end{document}